\documentclass{amsart}
\usepackage{amsmath,graphicx,latexsym}

\newtheorem{thm}{Theorem}
\newtheorem{lem}[thm]{Lemma}
\newtheorem{cor}[thm]{Corollary}

\theoremstyle{definition}
\newtheorem{defn}[thm]{Definition}
\newtheorem{rmk}[thm]{Remark}

\newcommand{\CPb}{\overline{\mathbb{CP}}{}^{2}}
\newcommand{\CP}{{\mathbb{CP}}{}^{2}}
\newcommand{\R}{\mathbb{R}}
\newcommand{\Z}{\mathbb{Z}}
\newcommand{\M}{M_{n}^{p}}
\newcommand{\sM}{M_{1}^{p}}
\newcommand{\ssM}{M_{1}^{1}}
\newcommand{\Mn}{M_{n}^{1}}
\newcommand{\Mz}{M_{n}^{0}}

\title{Exotic smooth structures on $S^2\times S^2$}

\author[A. Akhmedov]{Anar Akhmedov}
\address{School of Mathematics,
University of Minnesota,
Minneapolis, MN, 55455, USA}
\email{akhmedov@math.umn.edu}

\author[B. D. Park]{B. Doug Park}
\address{Department of Pure Mathematics,
University of Waterloo,
Waterloo, ON, N2L 3G1, Canada}
\email{bdpark@math.uwaterloo.ca}

\date{March 26, 2010.  Revised on June 23, 2010}


\begin{document}

\begin{abstract}
We construct an infinite family of mutually nondiffeomorphic irreducible
smooth structures on
the topological $4$-manifold $S^2\times S^2$.
\end{abstract}

\maketitle

\section{Introduction}

Let $M$\/ denote a closed smooth $4$-manifold.
$M$\/ is called \emph{irreducible}\/ if every connected sum decomposition of $M$\/
as $M=X\# Y$\/ implies that either $X$\/ or $Y$\/ is homeomorphic to the $4$-sphere $S^4$.
The following lemma will be useful for determining irreducibility.

\begin{lem}\label{lem: irreducibility}
Every closed smooth oriented simply connected spin\/ $4$-manifold with nontrivial Seiberg-Witten invariant is irreducible.
\end{lem}

\begin{proof}
Let $M$\/ be a closed smooth oriented simply connected spin $4$-manifold with nontrivial Seiberg-Witten invariant.
Suppose $M=X\# Y$\/ is a connected sum of two smooth $4$-manifolds $X$\/ and $Y$.  Then both $X$\/ and $Y$\/ are simply connected and the intersection forms of $X$\/ and $Y$\/ are both even.

If $b_2^{+}(X)$ and $b_2^{+}(Y)$ are both strictly positive, then the Seiberg-Witten invariant of $X\# Y$\/ is trivial (cf.\ \cite{witten}).  This contradiction shows that one of $b_2^{+}(X)$ and $b_2^{+}(Y)$ is $0$.
Without loss of generality, assume $b_2^{+}(X)=0$.  If $b_2(X)=b_2^{-}(X)>0$, then the intersection form of $X$\/ is a nontrivial negative definite form, so by Donaldson's theorem in \cite{donaldson}, it is isomorphic to the diagonal form $b_2(X)[ -1]$.  But this contradicts the fact that the intersection form of $X$\/ is even.
Thus we conclude that $b_2(X)=0$.  Since $X$\/ is simply connected, $X$\/ must be homeomorphic to $S^4$ by Freedman's theorem in \cite{freedman}.
\end{proof}

To state our results, it will be convenient to introduce the following terminology.

\begin{defn}\label{defn: infty property}
Let $M$\/ be a smooth $4$-manifold.  We say that $M$\/ has \emph{$\infty$-property}\/ if there exist an irreducible symplectic $4$-manifold and infinitely many mutually nondiffeomorphic irreducible nonsymplectic $4$-manifolds, all of which are homeomorphic to $M$.
\end{defn}

Let $S^2\times S^2$ denote the cartesian product of two $2$-spheres.
It was proved in \cite{AP: spin} that $(2k-1)(S^2\times S^2)$,
the connected sum of $2k-1$ copies of $S^2\times S^2$,
has $\infty$-property for every integer $k\geq 138$.
The main goal of this paper is to prove the following.

\begin{thm}\label{thm: main}
$S^2\times S^2$ has $\infty$-property.
\end{thm}

At the moment, $S^2\times S^2$ has the smallest Euler characteristic (four)
amongst all closed simply connected topological $4$-manifolds
that are known to possess more than one smooth structure.
The only closed simply connected $4$-manifolds with smaller Euler characteristic
are $S^4$ and the complex projective plane $\CP$.
In \cite{AP: II}, we show that $(2k-1)(S^2\times S^2)$ has $\infty$-property for every $k\geq 2$.
From Theorem~\ref{thm: main},
we also obtain infinitely many mutually nondiffeomorphic (albeit not irreducible) smooth structures on $\CP\#m\CPb$ for every $m\geq 2$ by blowing up $m-1$ points on our exotic $S^2 \times S^2$'s.
When $m=4$, we can deduce the following.

\begin{cor}\label{cor: metric}
There exist infinitely many mutually nondiffeomorphic\/
$4$-manifolds $\{Y_n \mid n=1,2,3,\dots \}$ that are all homeomorphic to\/ $\CP\# 4\CPb$ and satisfy the following.
\begin{itemize}
\item[{\rm (i)}]  Each $Y_n$ does not admit any Einstein metric.
\item[{\rm (ii)}]  Each $Y_n$ has negative Yamabe invariant.
\item[{\rm (iii)}]  On each $Y_n$, there does not exist any nonsingular solution
to the normalized Ricci flow for any initial metric.
\end{itemize}
\end{cor}

\begin{proof}
In Section~\ref{sec: s2xs2}, we construct infinitely many mutually nondiffeomorphic $4$-manifolds $\{\Mn \mid n=1,2,3,\dots \}$ that are all homeomorphic to $S^2\times S^2$ and have distinct nontrivial Seiberg-Witten invariants.
Let $Y_n=\Mn \# 3\CPb$.
Then by Freedman's theorem in \cite{freedman},
each $Y_n$ is homeomorphic to $\CP\# 4\CPb$.

Part (i) now follows from Theorem~3.3 in \cite{lebrun: ricci}
(by setting $X=\Mn$, $k=3$ and $\ell=0$ in the notation of \cite{lebrun: ricci}).
Part (ii) follows from \cite{lebrun: einstein}
(see the paragraph preceding Theorem~7 in \cite{lebrun: einstein}).
Part (iii) follows from Theorem~A in \cite{IRS}
(by setting $X=\Mn$ and $k=3$ in the notation of \cite{IRS}).
\end{proof}

We point out that the analogue of Corollary~\ref{cor: metric} for
$\CP\# m\CPb$, when $m=5,6,7,8$, has been proved in \cite{IRS,RS}.
The proof of Theorem~\ref{thm: main} is spread out in
Sections~\ref{sec: model}--\ref{sec: sw}.
In Section~\ref{sec: pi_1}, we also construct other families of $4$-manifolds with cyclic fundamental groups and having Euler characteristic equal to four.
Our overall strategy is to apply the `reverse engineering' technique of \cite{FPS} to
a suitably chosen nontrivial genus 2 surface bundle over a genus 2 surface.
In \cite{baykur: talk, baykur}, Baykur has used a similar method
to construct exotic smooth structures on $\CP\#5\CPb$.

\section{Model complex surface}
\label{sec: model}

Let $\Sigma_g$ denote a closed genus $g$\/ Riemann surface.
Let $\tau_1:\Sigma_2\to\Sigma_2$ be an elliptic involution with
two fixed points $\{z_0,z_1\}$ such that $\Sigma_2/\langle\tau_1\rangle = \Sigma_1$.  Let
$\tau_2 :\Sigma_3\to \Sigma_3$ be a fixed point free involution with
$\Sigma_3/\langle\tau_2\rangle=\Sigma_2$.  See Figure~\ref{fig: involutions} and also
Figure~\ref{fig: lagrangian tori} wherein $\tau_1$ is a 180 degree anti-clockwise rotation around the `center' point $z_1$.
There is a free $\Z/2$ action on the product $\Sigma_2\times \Sigma_3$ given by $\alpha(z,w)=(\tau_1(z),\tau_2(w))$ for $\alpha\neq0 \in \Z/2$, $z\in\Sigma_2$, and $w\in\Sigma_3$.
Let $X=(\Sigma_2\times\Sigma_3)/\langle\alpha\rangle$ denote the quotient manifold,
and let $q:\Sigma_2\times\Sigma_3\rightarrow X$\/ denote the quotient map.
The Euler characteristic and the Betti numbers of $X$\/ are
$e(X) =e(\Sigma_2\times\Sigma_3)/2= (-2)(-4)/2 =4$,
$b_1(X)=6$ and $b_2(X)=14$.
$X$\/ is a minimal complex surface of general type with $p_g=q=3$ and $K^2=8$
(cf.~\cite{hacon-pardini}).

\begin{figure}[ht]
\begin{center}
\includegraphics[scale=.55]{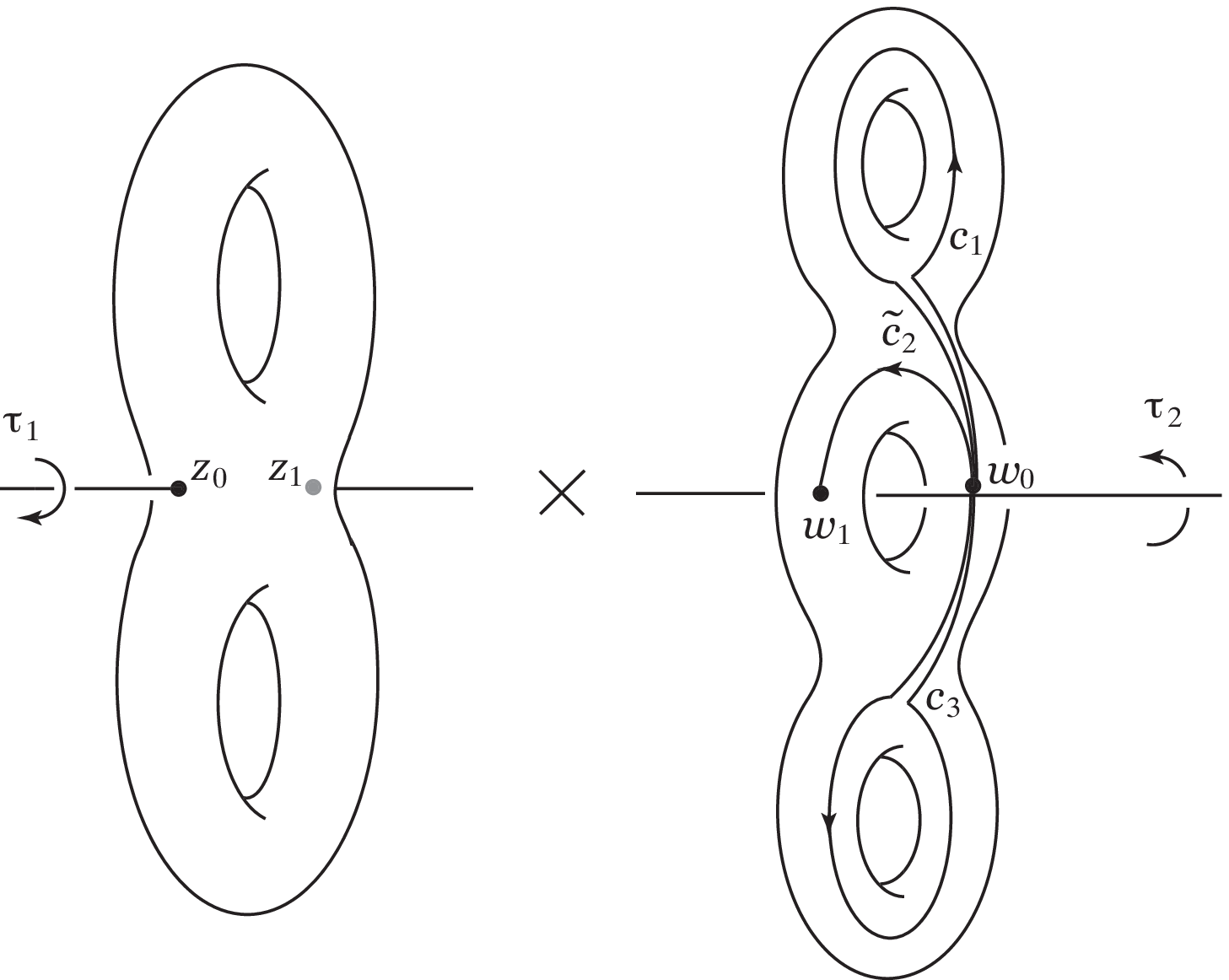}
\end{center}
\caption{Involution $\alpha=(\tau_1,\tau_2)$}
\label{fig: involutions}
\end{figure}

\begin{figure}[ht]
\begin{center}
\includegraphics[scale=.75]{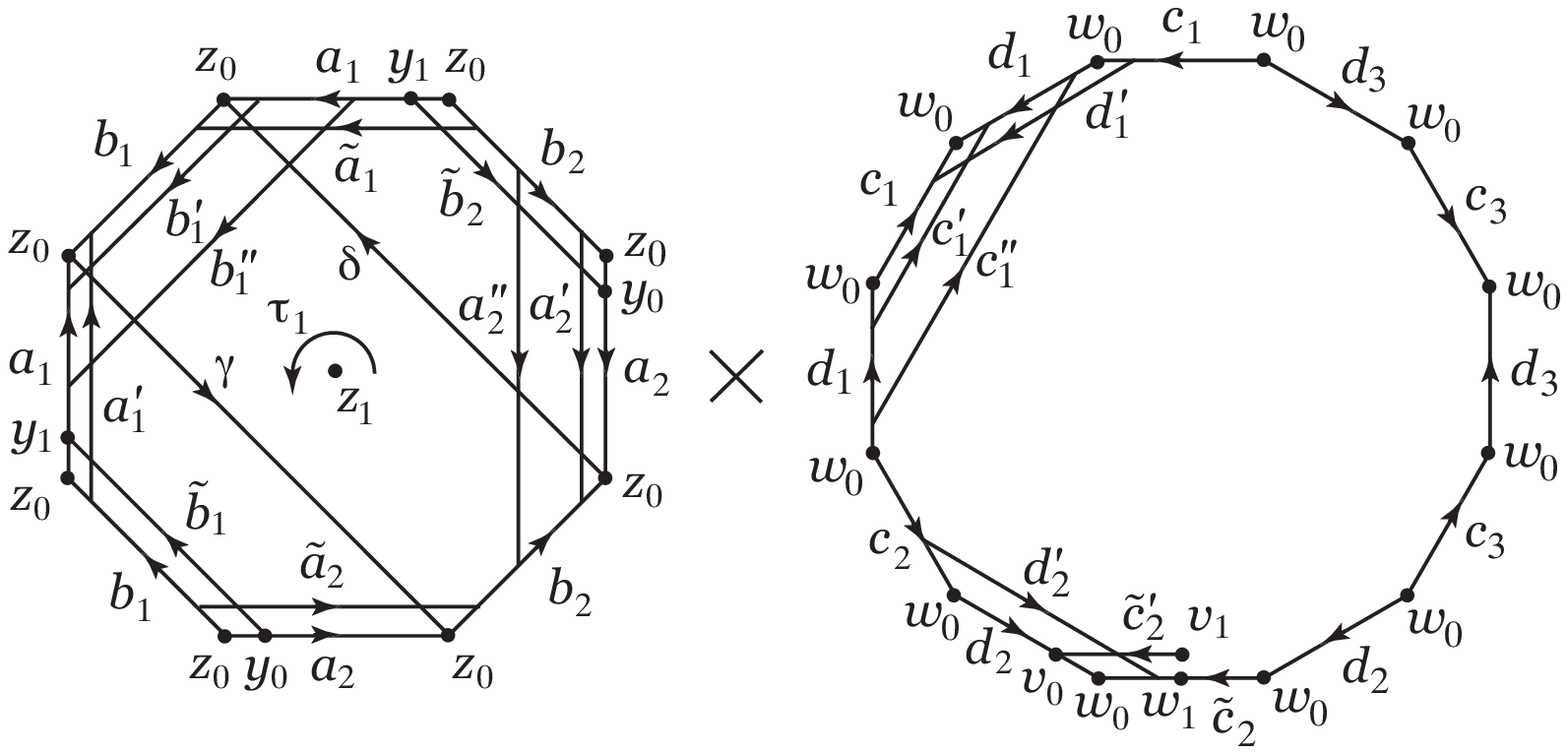}
\end{center}
\caption{Lifts of Lagrangian tori}
\label{fig: lagrangian tori}
\end{figure}

Let $\{a_1,b_1,a_2,b_2\}$ and $\{c_1,d_1,c_2,d_2,c_3,d_3\}$ be the
set of simple closed curves representing the standard generators of
$\pi_1(\Sigma_2,z_0)$ and $\pi_1(\Sigma_3,w_0)$, respectively.
Note that the base point $z_0$ of $\Sigma_2$ is one of the two fixed points of $\tau_1$.
Let $w_0$ be the base point of $\Sigma_3$ and
let $w_1=\tau_2(w_0)$ be as drawn in Figure~\ref{fig: involutions}.
Throughout, we choose $\{z_0\}\times\{w_0\}$ and $q(\{z_0\}\times\{w_0\})$ as the base points of $\pi_1(\Sigma_2\times\Sigma_3)$ and $\pi_1(X)$, respectively.
After isotopy and
by changing the orientations of the curves $a_2$ and $b_2$ if necessary,
we can assume that
$\alpha_{\ast}(a_1\times \{w_0\}) = a_2\times \{w_1\}$ and
$\alpha_{\ast}(b_1\times \{w_0\}) = b_2\times \{w_1\}$.

Let $\tilde{c}_2$ be a path from $w_0$ to $w_1$ in $\Sigma_3$ shown in Figures \ref{fig: involutions} and \ref{fig: lagrangian tori}.  Since $\tau_2$ maps the endpoints of $\tilde{c}_2$ to one another, $q(\{z_0\}\times\tilde{c}_2)$ is a closed path in $X$.
Note that $\tau_2(\tilde{c}_2)$ is a path from $w_1$ to $w_0$, and satisfy $q(\{z_0\}\times \tau_2(\tilde{c}_2))=q(\{z_0\}\times \tilde{c}_2)$ since $z_0$ is a fixed point of $\tau_1$.  It follows that $q_{\ast}(\{z_0\} \times \tilde{c}_2)^2=q_{\ast}(\{z_0\} \times c_2)$ in $\pi_1(X)$.
Similarly, $\tau_2(\tilde{c}_2^{-1})$ is a path from $w_0$ to $w_1$ that traverse along the `bottom' half of the loop $c_2$ in the clockwise direction in Figure~\ref{fig: involutions}.

Note that the loops $a_2\times\{w_0\}$ and $a_2\times\{w_1\}$ are freely homotopic along the path $\{z_0\}\times\tau_2(\tilde{c}_2^{-1})$ in $\Sigma_2\times\Sigma_3$.  Thus we have
\begin{equation}\label{eq: a's}
\begin{split}
q_{\ast}(a_2\times\{w_0\})&= q_{\ast}((\{z_0\}\times\tau_2(\tilde{c}_2^{-1}))\cdot(a_2\times\{w_1\})\cdot(\{z_0\}\times\tau_2(\tilde{c}_2)))
\\
&=q_{\ast}((\{z_0\}\times\tilde{c}_2^{-1})\cdot(a_1\times\{w_0\})\cdot(\{z_0\}\times\tilde{c}_2))
\end{split}
\end{equation}
inside $\pi_1(X)$.
In this paper, the order of path compositions will always be from left to right.
An explicit homotopy $F_1 : [0,1]\times [0,1]\to X$\/ is given by
\begin{equation}\label{eq: homotopy F_1}
F_1(s,t) = \left\{\begin{array}{lll}
q(a_2(0) \times \tau_2(\tilde{c}_2(1-3t))) & {\rm\ if\ } & 0\leq t \leq s/3 , \\[2pt]
q(a_2(\frac{3t-s}{3-2s}) \times \tau_2(\tilde{c}_2(1-s)))  &  {\rm \ if\ } & s/3\leq t \leq (3-s)/3 , \\[2pt]
q(a_2(1) \times \tau_2(\tilde{c}_2(3t-2))) & {\rm \ if\ } & (3-s)/3 \leq t \leq 1,
\end{array}\right.
\end{equation}
where $a_2(t)$ and $\tilde{c}_2(t)$ are parameterizations of the curves $a_2$ and $\tilde{c}_2$ satisfying $a_2(0)=a_2(1)=z_0$, $\tilde{c}_2(0)=w_0$, and $\tilde{c}_2(1)=w_1$.
Similarly, we have
\begin{equation}\label{eq: b's}
q_{\ast}(b_2\times\{w_0\})=
q_{\ast}((\{z_0\}\times\tilde{c}_2^{-1})\cdot(b_1\times\{w_0\})\cdot(\{z_0\}\times\tilde{c}_2))
\end{equation}
inside $\pi_1(X)$ via the based homotopy
\begin{equation}\label{eq: homotopy F_2}
F_2(s,t) = \left\{\begin{array}{lll}
q(b_2(0) \times \tau_2(\tilde{c}_2(1-3t))) & {\rm\ if\ } & 0\leq t \leq s/3 , \\[2pt]
q(b_2(\frac{3t-s}{3-2s}) \times \tau_2(\tilde{c}_2(1-s)))  &  {\rm \ if\ } & s/3\leq t \leq (3-s)/3 , \\[2pt]
q(b_2(1) \times \tau_2(\tilde{c}_2(3t-2))) & {\rm \ if\ } & (3-s)/3 \leq t \leq 1.
\end{array}\right.
\end{equation}

By using based homotopies supported inside $q(\{z_0\}\times\Sigma_3)$,
we also deduce that
\begin{align*}
q_{\ast}(\{z_0\}\times c_3)
&=  q_{\ast}((\{z_0\}\times \tau_2(\tilde{c}_2^{-1}))\cdot( \{z_0\}\times \tau_2(c_1))\cdot(\{z_0\}\times \tau_2(\tilde{c}_2))) \\
&=  q_{\ast}((\{z_0\}\times \tilde{c}_2^{-1})\cdot( \{z_0\}\times c_1)\cdot(\{z_0\}\times \tilde{c}_2)), \\
q_{\ast}(\{z_0\}\times d_3)
&= q_{\ast}((\{z_0\}\times \tau_2(\tilde{c}_2^{-1}))\cdot(\{z_0\}\times \tau_2(d_1))\cdot(\{z_0\}\times \tau_2(\tilde{c}_2))) \\
&= q_{\ast}((\{z_0\}\times\tilde{c}_2^{-1})\cdot(\{z_0\}\times  d_1)\cdot(\{z_0\}\times \tilde{c}_2)).
\end{align*}
Note again that, to go from $w_0$ to $w_1$, we have used the `bottom' half of the loop $c_2$ and traversed it in the clockwise direction in Figure~\ref{fig: involutions}.  This is why we have $q(\{z_0\}\times\tau_2(\tilde{c}_2^{-1}))=q(\{z_0\}\times\tilde{c}_2^{-1})$ first.

The quotient group $\pi_1(X)/q_{\ast}(\pi_1(\Sigma_2\times\Sigma_3))$ is isomorphic to $\Z/2$, and is generated by the coset of $q(\{z_0\}\times \tilde{c}_2)$.
In summary, we have proved the following.

\begin{lem}\label{lem: generators}
The following six loops generate $\pi_1(X)${\rm :}
\begin{gather*}
q(a_1\times\{w_0\}),\ \  q(b_1\times\{w_0\}), \ \ q(\{z_0\}\times c_1),\\
q(\{z_0\}\times d_1),\ \  q(\{z_0\}\times \tilde{c}_2), \ \ q(\{z_0\}\times d_2).
\end{gather*}
\end{lem}

From now on, we will sometimes abuse notation and write $a_1=q(a_1\times\{w_0\})$ and $c_1=q(\{z_0\}\times c_1)$, etc.
The intersection form of $X$\/ is isomorphic to $7H$, where
\begin{equation}\label{eq: H}
H = \left[\begin{array}{cc}
0&1 \\
1&0
\end{array}
\right].
\end{equation}
Hence $\sigma(X)$, the signature of $X$, is $0$.
A basis for the intersection form of $X$\/ is given by the following seven geometrically dual pairs:
\begin{gather*}
([a_1 \times c_1], -[b_1 \times d_1]), \ \ ([a_1 \times d_1], [b_1 \times c_1]),\\
([a_2 \times c_1], -[b_2 \times d_1]), \ \ ([a_2 \times d_1], [b_2 \times c_1]),\\
([(\tilde{a}_1\tilde{a}_2) \times \tilde{c}_2], -[b_1 \times d_2]), \ \ ([a_1 \times d_2], [(\tilde{b}_1\tilde{b}_2) \times \tilde{c}_2]), \\
([\Sigma_2 \times \{w_0\}], [\{z_0\} \times \Sigma_3]).
\end{gather*}
Here, $[\,\cdot\,]$ denotes the homology class of the image $q(\,\cdot\,)$ in the quotient manifold $X$\/ for short.  Note that even though $[a_1\times\{w_0\}]=[a_2\times\{w_0\}]$ in $H_1(X;\Z)$,  we still have $[a_1 \times c_1]\neq
[a_2 \times c_1]$ since $\alpha_{\ast}[a_1 \times c_1] = [a_2\times c_3]$ in $H_2(\Sigma_2\times \Sigma_3;\Z)$.
The minus signs are there to ensure that the nonzero intersection numbers are $+1$ with respect to standard orientations.  For example,
\begin{align*}
[a_1 \times c_1] \cdot [b_1 \times d_1] &=
(-1)^{\deg(c_1)\deg(b_1)}(a_1\cdot b_1)(c_1\cdot d_1) \\
&= -1\cdot 1\cdot 1 = -1, \\
[a_1 \times d_1] \cdot [b_1 \times c_1] &=
(-1)^{\deg(d_1)\deg(b_1)}(a_1\cdot b_1)(d_1\cdot c_1) \\
&= (-1)\cdot 1 \cdot (-1) = 1.
\end{align*}

The composite loop $\tilde{b}_1\tilde{b}_2$ starts at a point $y_0$ on $a_2$ and traverses first along $\tilde{b}_1$ to the point $y_1=\tau_1(y_0)$ on $a_1$
and then along $\tilde{b}_2$ back to the starting point $y_0$.  See Figure~\ref{fig: lagrangian tori}.
We define the loop $\tilde{a}_1\tilde{a}_2$ in a similar manner.
Both loops $\tilde{a}_1\tilde{a}_2$ and $\tilde{b}_1\tilde{b}_2$ are mapped to themselves under $\tau_1$.  Since the endpoints of $\tilde{c}_2$ are mapped to each other under $\tau_2$, the cylinders $(\tilde{a}_1\tilde{a}_2) \times \tilde{c}_2$ and $(\tilde{b}_1\tilde{b}_2) \times \tilde{c}_2$ become closed tori in $X$.

As observed in Example~2 of \cite{hacon-pardini}, it is convenient to view $X$\/ as
the total space of a genus 2 surface bundle over a genus 2 surface:
\begin{equation}\label{eq: bundle}
X=\frac{\Sigma_2\times\Sigma_3}{\langle\alpha\rangle}
\stackrel{f}{\longrightarrow} \frac{\Sigma_3}{\langle\tau_2\rangle} =\Sigma_2.
\end{equation}
The quotients $q(\Sigma_2 \times \{w_0\})$ and $q(\{z_0\} \times \Sigma_3)$ are genus 2 surfaces in $X$ that form a fiber and a section of this bundle.
Thus $([\Sigma_2 \times \{w_0\}], [\{z_0\} \times \Sigma_3])$ is represented by a pair of transversely intersecting genus 2 symplectic surfaces in $X$.

\begin{lem}\label{lem: infinite order}
The loops $q(\{z_0\}\times c_1)$,
$q(\{z_0\}\times d_1)$, $q(\{z_0\}\times \tilde{c}_2)$ and $q(\{z_0\}\times d_2)$
represent elements of infinite order in $\pi_1(X)$.
\end{lem}

\begin{proof}
From the homotopy long exact sequence for a fibration, we obtain
\begin{equation*}
0 \longrightarrow \pi_1(\Sigma_2) \longrightarrow \pi_1(X)
\stackrel{f_{\ast}}{\longrightarrow} \pi_1(\Sigma_2) \longrightarrow 0,
\end{equation*}
where $f_{\ast}$ is the homomorphism induced by the bundle map
in (\ref{eq: bundle}).
$f_{\ast}$ maps the loops in our lemma to the standard generators of
$\pi_1(\Sigma_2)$ and hence these loops cannot be torsion elements of
$\pi_1(X)$.
\end{proof}

\section{Construction of exotic $S^2\times S^2$}
\label{sec: s2xs2}

Choose $\tau_1$ and $\tau_2$ invariant volume forms on $\Sigma_2$ and $\Sigma_3$, respectively.
By pushing forward the sum of pullbacks of these volume forms on $\Sigma_2\times\Sigma_3$ under $q$,
we obtain a symplectic form $\omega$\/ on $X$.
Alternatively, we can equip $X$\/ with a symplectic form coming from the bundle structure in (\ref{eq: bundle}).
Now consider the following six Lagrangian tori in $X$:
\begin{equation}\label{eq: 6 tori}
\begin{array}{ccc}
q(a_1' \times c_1'),  &
q(b_1' \times c_1''), &
q(a_2' \times c_1'),\\[2pt]
q(a_2'' \times d_1'), &
q(b_1'' \times d_2'), &
q((\tilde{b}_1\tilde{b}_2) \times \tilde{c}_2').
\end{array}
\end{equation}
The lifts of these Lagrangian tori in $\Sigma_2\times \Sigma_3$ are drawn in Figure~\ref{fig: lagrangian tori}.
The prime and double prime notations are explained in \cite{FPS}.
Each tori in (\ref{eq: 6 tori}) is Lagrangian with respect to $\omega$\/ since
the first circle factor lies in the $\Sigma_2$ direction
whereas the second circle factor lies in the $\Sigma_3$ direction.

The $\tilde{c}_2'$ path that is `parallel' to $\tau_2(\tilde{c}_2)$ is drawn in Figures \ref{fig: lagrangian tori} and \ref{fig: c_2}.
The endpoints of $\tilde{c}_2'$, $v_0$ and $v_1$, are mapped to each other by $\tau_2$.  Thus the cylinder $(\tilde{b}_1\tilde{b}_2) \times \tilde{c}_2'$ in $\Sigma_2\times\Sigma_3$ becomes a closed torus in $X$.
Let $\tilde{b}_1(t)$, $\tilde{b}_2(t)$ and $\tilde{c}_2'(t)$ be parameterizations of the curves $\tilde{b}_1$,
$\tilde{b}_2$ and $\tilde{c}_2'$, respectively, satisfying $\tilde{b}_1(0)=y_0$, $\tilde{b}_1(1)=y_1$,
$\tilde{b}_2(0)=y_1$, $\tilde{b}_2(1)=y_0$, $\tilde{c}_2'(0)=v_1$ and $\tilde{c}_2'(1)=v_0$.  
Note that $q(\{z\}\times \tilde{c}_2')$ is not a closed loop for any point 
$z\in \tilde{b}_1\tilde{b}_2$ since there is no fixed point of $\tau_1$ 
on $\tilde{b}_1\tilde{b}_2$.  
However, the following composition of paths gives rise to a simple closed curve on 
the $q((\tilde{b}_1\tilde{b}_2) \times \tilde{c}_2')$ torus:
\begin{equation}\label{eq: beta}
\beta(t) = \left\{\begin{array}{lll}
\{y_1\} \times \tilde{c}_2'(2t) & {\rm\ if\ } & 0\leq t \leq 1/2, \\[2pt]
\tilde{b}_2(2t-1) \times \{v_0\}  &  {\rm \ if\ } & 1/2\leq t \leq 1.
\end{array}\right.
\end{equation}
Since $q(\{y_1\} \times \{v_1\})=q(\{y_0\} \times \{v_0\})$ when $t=0,1$, the loop
$q(\beta(t))$ is well defined.  

Let $\eta$\/ be a short `diagonal' path on the $a_1\times \tau_2(d_2)$ torus from 
$\{z_0\} \times \{w_1\}$ to $\{y_1\} \times \{v_1\}$.  
Then $\alpha(\eta)$ is a path on the $a_2\times d_2$ torus from $\{z_0\} \times \{w_0\}$ to $\{y_0\} \times \{v_0\}$.  
Since $q(\{z_0\} \times \{w_1\})=q(\{z_0\} \times \{w_0\})$, 
the image $q(\eta \cdot \beta \cdot \alpha(\eta)^{-1})$ represents 
$\tilde{c}_2 b_2$ in $\pi_1(X)$.
Next consider the composition
\begin{equation}\label{eq: xi}
\xi(t) = \left\{\begin{array}{lll}
\{y_1\} \times \tilde{c}_2'(4t) & {\rm\ if\ } & 0\leq t \leq 1/4, \\[2pt]
\tilde{b}_2(4t-1) \times \{v_0\}  &  {\rm \ if\ } & 1/4\leq t \leq 1/2, \\[2pt]
\tilde{b}_1(4t-2) \times \{v_0\}  &  {\rm \ if\ } & 1/2\leq t \leq 3/4, \\[2pt]
\{y_1\} \times \tilde{c}_2'(4-4t)  &  {\rm \ if\ } & 3/4\leq t \leq 1.
\end{array}\right.
\end{equation}
Note that $q(\eta \cdot \xi \cdot \eta^{-1})$ represents $\tilde{c}_2 b_2 b_1 \tilde{c}_2^{-1}
= b_1 b_2$ in $\pi_1(X)$.  

Both paths $\beta$\/ and $\xi$\/ begin at the point $\{y_1\}\times\{v_1\}$,
and their images under $q$\/ represent standard generators for the image of the fundamental group of the
$q((\tilde{b}_1\tilde{b}_2) \times \tilde{c}_2')$ torus in $X$\/ that are
based at $q(\{y_1\} \times \{v_1\})$.
In particular, the words for the $q((\tilde{b}_1\tilde{b}_2) \times \tilde{c}_2')$ torus
in $\pi_1(X)$ are given by the conjugates of
\begin{equation*}
(\tilde{c}_2 b_2)^{-1} (\tilde{c}_2 b_2 b_1 \tilde{c}_2^{-1}) (\tilde{c}_2 b_2)
(\tilde{c}_2 b_2 b_1 \tilde{c}_2^{-1})^{-1} =
b_1 b_2 \tilde{c}_2 b_1^{-1} b_2^{-1} \tilde{c}_2^{-1}.
\end{equation*}

\begin{figure}[ht]
\begin{center}
\includegraphics[scale=.55]{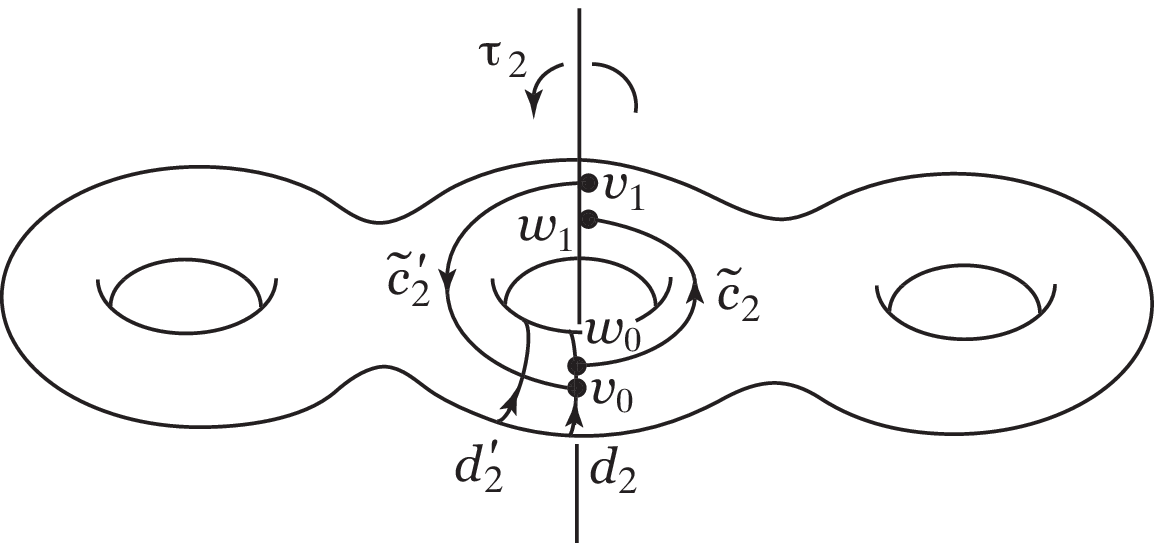}
\end{center}
\caption{$\tilde{c}_2'$ and $d_2'$}
\label{fig: c_2}
\end{figure}

In what follows, for the sake of brevity,
we will sometimes abuse notation and blur the distinction between a surface or a curve
in $\Sigma_2\times \Sigma_3$ and its image under $q$\/ in the quotient manifold $X$.
We will use the notation $[g,h]=ghg^{-1}h^{-1}$ for the commutator.

\begin{lem}\label{lem: framings}
Let $X_0$ denote the complement of tubular neighborhoods of the six Lagrangian tori of\/ $(\ref{eq: 6 tori})$ in $X$.
The Lagrangian framings give the following bases for the images of the fundamental groups of the $3$-torus components of the boundary $\partial X_0$.
\begin{equation}\label{eq: framings}
\begin{array}{cc}
\{ a_1 ,\, c_1 ;\, [b_1^{-1}, d_1^{-1}] \}, &
\{ b_1 ,\, d_1 c_1 d_1^{-1} ;\, [a_1^{-1}, d_1] \}, \\[2pt]
\{ a_2 ,\, c_1 ;\, [b_2^{-1}, d_1^{-1}] \}, &
\{ b_2 a_2 b_2^{-1} ,\, d_1 ;\, [b_2, c_1^{-1}] \}, \\[2pt]
\{ a_1 b_1 a_1^{-1} ,\, d_2 ;\,
\tilde{c}_2^{-1}  a_1 a_2 \tilde{c}_2 a_1^{-1} a_2^{-1} \}, &
\{ \tilde{c}_2 b_2 b_1 \tilde{c}_2^{-1},\, \tilde{c}_2 b_2 ;\, [a_1, \tilde{c}_2 d_2^{-1} \tilde{c}_2^{-1}] \}.
\end{array}
\end{equation}
\end{lem}

\begin{proof}
The point $q(\{z_0\}\times\{w_0\})$ lies in $X_0$ and 
we choose it for the base point of $\pi_1(X_0)$.  
The first four triples in (\ref{eq: framings}) are in standard form 
and can be derived as in \cite{FPS}.
For the fifth triple corresponding to $q(b_1'' \times d_2')$, the Lagrangian push-offs of $b_1''$ and $d_2'$ represent $a_1 b_1 a_1^{-1}$ and $d_2$, respectively, by a standard argument in \cite{FPS}.
The orientation convention for the boundary $\partial X_0$ dictates that the third member of our triple (after the semicolon) should be the `clockwise' meridian of $q(b_1'' \times d_2')$.
(The anti-clockwise meridian is usually reserved for the boundary of the tubular neighborhood of $q(b_1'' \times d_2')$.)
Note that $q(b_1'' \times d_2')$ torus intersects $q((\tilde{a}_1\tilde{a}_2) \times \tilde{c}_2)$ torus once negatively in $X$.
Hence the clockwise meridian of $q(b_1'' \times d_2')$ is given by a word for the punctured $q((\tilde{a}_1\tilde{a}_2) \times \tilde{c}_2)$ torus, read in the anti-clockwise direction.
To reach $b_1'' \times d_2'$ torus from the preimage of the base point $\{z_0\}\times\{w_0\}$, we need to travel positively in the $a_1\times\{w_0\}$ direction and negatively in the $\{z_0\} \times \tau_2(\tilde{c}_2)$ direction.
As explained in \cite{FPS}, the clockwise meridian of $q(b_1'' \times d_2')$ 
is then given by $q(\tilde{c}_2^{-1}(a_1a_2)\tilde{c}_2(a_2a_1)^{-1})$, coming from the punctured $q((\tilde{a}_1\tilde{a}_2) \times \tilde{c}_2)$ torus.  See the left half of Figure~\ref{fig: torus}.

\begin{figure}[ht]
\begin{center}
\includegraphics[scale=.65]{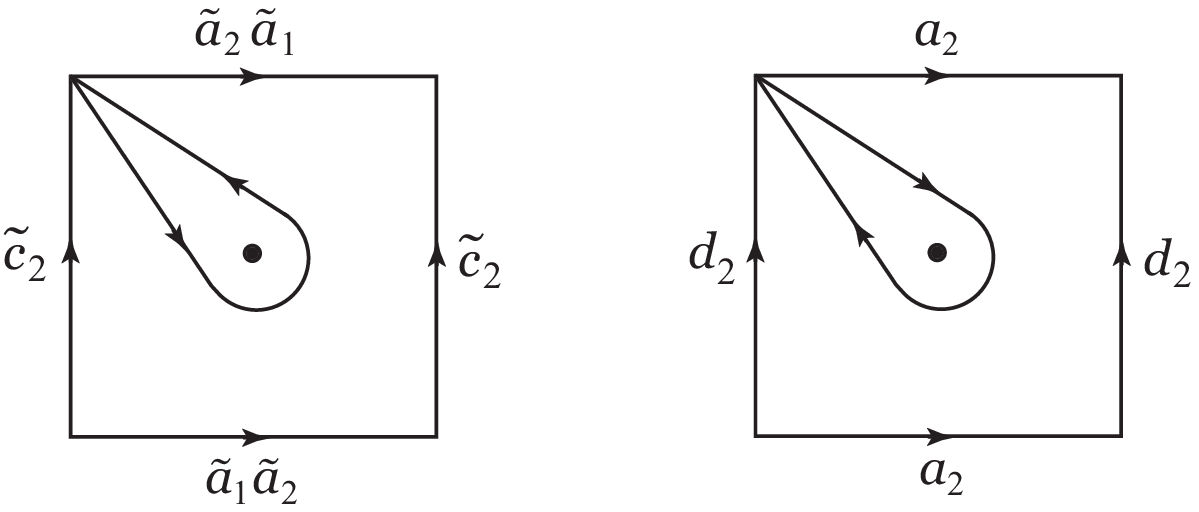}
\end{center}
\caption{Punctured $q((\tilde{a}_1\tilde{a}_2) \times \tilde{c}_2)$ and
$q(a_2\times d_2)$ tori}
\label{fig: torus}
\end{figure}

For the sixth triple, it is clear that the Lagrangian push-offs of
$q(\beta)$ and $q(\xi)$ (see (\ref{eq: beta}) and (\ref{eq: xi})) represent
the homotopy classes of
\begin{equation*}
q((\{z_0\}\times\tau_2(\tilde{c}_2))\cdot(b_2 \times\{w_0\}))
=q(\{z_0\}\times \tilde{c}_2)\cdot q(b_2 \times\{w_0\}),
\end{equation*}
and
\begin{align*}
&q((\{z_0\}\times\tau_2(\tilde{c}_2))\cdot (b_2 \times\{w_0\})\cdot
(b_1 \times\{w_0\})\cdot (\{z_0\}\times \tau_2(\tilde{c}_2^{-1})))\\
&=q(\{z_0\}\times \tilde{c}_2)\cdot q(b_2 \times\{w_0\})\cdot
q(b_1 \times\{w_0\})\cdot q(\{z_0\}\times \tilde{c}_2^{-1}),
\end{align*}
respectively, in $\pi_1(X_0)$.
Since $\beta$\/ and $\xi$\/ both start at $\{y_1\}\times\{v_1\}$, 
the lift of our meridian must start at the nearby preimage of the basepoint $\{z_0\}\times\{w_1\}$, 
rather than starting at $\{z_0\}\times\{w_0\}$.  
(Note that $\{z_0\}\times\{w_0\}$ lies near $\{y_0\}\times\{v_0\}$.)  

The $(\tilde{b}_1\tilde{b}_2) \times \tilde{c}_2'$ cylinder intersects the 
$a_2\times d_2$ torus once positively at $\{y_0\}\times\{v_0\}$.  
Thus a lift of the clockwise meridian of $q((\tilde{b}_1\tilde{b}_2) \times \tilde{c}_2')$, that 
starts at $\{z_0\}\times\{w_0\}$, is given by a word for the punctured $a_2\times d_2$ torus, read in the clockwise direction.
To reach $(\tilde{b}_1\tilde{b}_2) \times \tilde{c}_2'$ cylinder from the point $\{z_0\}\times\{w_0\}$, we need to travel positively in the $a_2\times\{w_0\}$ direction and negatively in the $\{z_0\} \times d_2$ direction.  
Hence a lift of the clockwise meridian of $q((\tilde{b}_1\tilde{b}_2) \times \tilde{c}_2')$ torus, 
that starts at $\{z_0\}\times\{w_0\}$, is given by $a_2 d_2^{-1} a_2^{-1} d_2=[a_2, d_2^{-1}]$, 
coming from the punctured $a_2\times d_2$ torus.
See the right half of Figure~\ref{fig: torus}.  

Now the path $\{z_0\}\times \tau_2(\tilde{c}_2)$ starts at $\{z_0\}\times\{w_1\}$,
ends at $\{z_0\}\times\{w_0\}$, and is very close and parallel to the $\{y_1\} \times \tilde{c}_2'$ part of the 
$q((\tilde{b}_1\tilde{b}_2) \times \tilde{c}_2')$ torus.  
Since the required lift of our meridian must begin at $\{z_0\}\times\{w_1\}$, 
this lift must be the conjugate of $[a_2, d_2^{-1}]$ by $\{z_0\}\times \tau_2(\tilde{c}_2)$, 
whose image in $\pi_1(X_0)$ is given by 
$\tilde{c}_2 [a_2, d_2^{-1}]\tilde{c}_2^{-1}=[a_1, \tilde{c}_2 d_2^{-1} \tilde{c}_2^{-1}]$.  
Note that the identity $a_2 = \tilde{c}_2^{-1} a_1 \tilde{c}_2$ continues to hold in $\pi_1(X_0)$ 
since the image of homotopy (\ref{eq: homotopy F_1}) is contained in $X_0$ 
(cf.~the proof of Lemma~\ref{lem: presentation} below).  
\end{proof}

Let $n\geq 1$ and $p\geq 0$ be a pair of integers.
Inside $X$, we perform the following six torus surgeries:
\begin{equation}\label{eq: 6 surgeries}
\begin{array}{cc}
(a_1' \times c_1', a_1', -n), &
(b_1' \times c_1'', b_1', -1), \\[2pt]
(a_2' \times c_1', c_1', -1), &
(a_2'' \times d_1', d_1', -1),\\[2pt]
(b_1'' \times d_2', d_2', -1/p), &
((\tilde{b}_1\tilde{b}_2) \times \tilde{c}_2', \beta, -1).
\end{array}
\end{equation}
Here, we are using the notation from \cite{FPS, ABP}.  
For example, the first surgery is a $(-n)$-surgery on $q(a_1' \times c_1')$ torus
along $q(a_1')$ loop
with respect to the Lagrangian framing in (\ref{eq: framings}).
The first and the fifth surgeries are Luttinger surgeries (cf.~\cite{luttinger, ADK}) when $n=1$ and $p\geq 1$, respectively.
If $p=0$, then the fifth surgery is trivial, i.e., the surgery does not alter the $4$-manifold.
The other four surgeries with $-1$ coefficient are all Luttinger surgeries.

Let $\M$ denote the resulting closed $4$-manifold after the surgeries in (\ref{eq: 6 surgeries}).
If $p\geq 1$, then we have
\begin{equation*}
b_1(\M)=b_1(X)-6=0 \quad \text{and} \quad
b_2(\M)=b_2(X)-2\cdot 6=2.
\end{equation*}
The intersection form of $\M$ is even for every $n\geq 1$ and $p\geq 0$.
Note that
$\sM$ is a minimal symplectic $4$-manifold for every $p\geq 0$.

\section{Calculation of fundamental group}
\label{sec: pi_1}

The point $q(\{z_0\}\times\{w_0\})$ lies in $X_0\subset \M$ and hence we will choose it for the base point of $\pi_1(\M)$.

\begin{lem}\label{lem: presentation}
$\pi_1(\M)$ is generated by $a_1,b_1,a_2,b_2,c_1,d_1,\tilde{c}_2,d_2$.
The following relations hold in $\pi_1(\M)${\rm :}
\begin{gather*}
a_2=\tilde{c}_2^{-1} a_1\tilde{c}_2, \ \
b_2=\tilde{c}_2^{-1} b_1 \tilde{c}_2, \ \
b_1=\tilde{c}_2^{-1} b_2 \tilde{c}_2, \\
[b_2,d_2]=1, \ \
[a_1^{-1}b_1^{-1}a_2,d_2]=1, \ \
[a_2^{-1}b_2^{-1}a_1,d_2]=1, \\
[b_1^{-1}, d_1^{-1}]^n=a_1,\ \  [a_1^{-1}, d_1]=b_1,\\
[b_2^{-1}, d_1^{-1}]=c_1,\ \
[b_2, c_1^{-1}]=d_1,\\
\tilde{c}_2^{-1} a_1 a_2 \tilde{c}_2 a_1^{-1} a_2^{-1}=d_2^p, \ \
[a_1, \tilde{c}_2 d_2^{-1} \tilde{c}_2^{-1}]=\tilde{c}_2 b_2,\\
[a_1,c_1]=1, \ \ [b_1,c_1]=1,\ \ [a_2,c_1]=1, \ \
[a_2,d_1]=1, \ \ [b_1,d_2]=1,\\
[a_1,b_1][a_2,b_2]=1, \ \
[c_1,d_1][\tilde{c}_2,d_2]=1.
\end {gather*}
\end{lem}

\begin{proof}
By using Seifert-Van Kampen theorem,
the generators of $\pi_1(\M)$ can be determined from Lemmas~\ref{lem: generators} and \ref{lem: framings}.
The first relation comes from (\ref{eq: a's}), which continues to hold in $\pi_1(\M)$ because the image of homotopy (\ref{eq: homotopy F_1}) is contained in $q(a_2\times\tau_2(\tilde{c}_2))$, which lies inside $X_0 \subset \M$.
For example, we see that $q(a_2\times\tau_2(\tilde{c}_2))$ is disjoint from the fifth surgery torus $q(b_1'' \times d_2')$ in (\ref{eq: 6 surgeries}), since $a_2\times \tau_2(\tilde{c}_2)$ is disjoint from $(b_1'' \times d_2') \cup (\tau_1(b_1'') \times \tau_2(d_2'))$ in $\Sigma_2\times \Sigma_3$.  See Figures~\ref{fig: lagrangian tori} and \ref{fig: c_2}.

The second relation comes from (\ref{eq: b's}), which continues to hold in $\pi_1(\M)$ because the image of homotopy (\ref{eq: homotopy F_2}) is contained in $q(b_2\times\tau_2(\tilde{c}_2)) \subset X_0$.  The third relation can be written as
\begin{equation*}
\begin{split}
q_{\ast}(b_1\times\{w_0\})&=q_{\ast}((\{z_0\}\times\tilde{c}_2^{-1})\cdot(b_2\times\{w_0\})\cdot(\{z_0\}\times\tilde{c}_2)) \\
&= q_{\ast}((\{z_0\}\times\tau_2(\tilde{c}_2^{-1}))\cdot(b_1\times\{w_1\})\cdot(\{z_0\}\times\tau_2(\tilde{c}_2)))
\end{split}
\end{equation*}
The corresponding based homotopy is given by
\begin{equation*}
F_3(s,t) = \left\{\begin{array}{lll}
q(b_1(0) \times \tau_2(\tilde{c}_2(1-3t))) & {\rm\ if\ } & 0\leq t \leq s/3 , \\[2pt]
q(b_1(\frac{3t-s}{3-2s}) \times \tau_2(\tilde{c}_2(1-s)))  &  {\rm \ if\ } & s/3\leq t \leq (3-s)/3 , \\[2pt]
q(b_1(1) \times \tau_2(\tilde{c}_2(3t-2))) & {\rm \ if\ } & (3-s)/3 \leq t \leq 1.
\end{array}\right.
\end{equation*}
The image of $F_3$ is contained in $q(b_1\times\tau_2(\tilde{c}_2))$,
which in turn lies inside $X_0$.

The fourth relation comes from the torus $q(b_2\times d_2)$ lying in $X_0$.
The fifth relation comes from the torus $q(\gamma\times d_2)$ lying in $X_0$,
where $\gamma$\/ is the closed path drawn in Figure~\ref{fig: lagrangian tori}.
Note that
\begin{equation*}
q_{\ast}(\gamma\times\{w_0\})=q_{\ast}((a_1^{-1}\times\{w_0\})
\cdot(b_1^{-1}\times\{w_0\})\cdot(a_2\times\{w_0\}))
\end{equation*}
via a based homotopy that is supported inside $q(\Sigma_2\times\{w_0\})\subset X_0$.
Similarly, the sixth relation comes from the torus $q(\delta\times d_2)$ in $X_0$,
where $\delta$\/ is the closed path drawn in Figure~\ref{fig: lagrangian tori} satisfying
\begin{equation*}
q_{\ast}(\delta\times\{w_0\})=q_{\ast}((a_2^{-1}\times\{w_0\})
\cdot(b_2^{-1}\times\{w_0\})\cdot(a_1\times\{w_0\})).
\end{equation*}

The next eleven relations are standard and can be derived as in \cite{FPS} from Lemma~\ref{lem: framings} and the definition of torus surgery.
Note that $q_{\ast}([a_1,b_1][a_2,b_2] \times \{w_0\})=1$ holds in $\pi_1(\M)$
because of the presence of the genus 2 surface
$q(\Sigma_2 \times \{w_0\})$ inside $X_0 \subset \M$.
The last relation $q_{\ast}(\{z_0\}\times [c_1,d_1][\tilde{c}_2,d_2])=1$ holds in $\pi_1(\M)$ since $q(\{z_0\}\times\Sigma_3)$ is a genus 2 surface in $X_0 \subset \M$ whose fundamental group is generated by
$q(\{z_0\}\times c_1)$, $q(\{z_0\}\times d_1)$, $q(\{z_0\}\times\tilde{c}_2)$ and $q(\{z_0\}\times d_2)$.
\end{proof}

\begin{thm}\label{thm: pi_1}
We have $\pi_1(\M)\cong\Z/p$.
In particular, $\pi_1(\Mz)\cong\Z$ and
$\pi_1(\Mn)=0$ for every integer $n\geq 1$.
\end{thm}

\begin{proof}
From the fifth and sixth relations in Lemma~\ref{lem: presentation}, we deduce that
$d_2^{-1}$ commutes with the product
\begin{equation*}
(a_2^{-1}b_2^{-1}a_1)(a_1^{-1}b_1^{-1}a_2)= a_2^{-1}b_2^{-1}b_1^{-1}a_2.
\end{equation*}
It follows that
\begin{equation}\label{eq: commutator}
[b_2^{-1}b_1^{-1},a_2d_2^{-1}a_2^{-1}]=
a_2[ a_2^{-1}b_2^{-1}b_1^{-1}a_2,d_2^{-1}]a_2^{-1}=1.
\end{equation}
From the first and the twelfth relations in Lemma~\ref{lem: presentation},
we have
\begin{equation*}
\tilde{c}_2 b_2 = a_1 \tilde{c}_2 d_2^{-1} \tilde{c}_2^{-1}
a_1^{-1} \tilde{c}_2 d_2 \tilde{c}_2^{-1} =
\tilde{c}_2 a_2 d_2^{-1} a_2^{-1} d_2 \tilde{c}_2^{-1}.
\end{equation*}
Canceling the $\tilde{c}_2$'s from both sides and then rearranging, we conclude that
\begin{equation*}
b_2 \tilde{c}_2 d_2^{-1} = a_2 d_2^{-1} a_2^{-1} .
\end{equation*}
Thus (\ref{eq: commutator}) can be rewritten as
\begin{equation*}
1=[b_2^{-1}b_1^{-1}, b_2 \tilde{c}_2 d_2^{-1}]=
b_2^{-1}b_1^{-1} b_2 \tilde{c}_2 d_2 ^{-1} b_1 b_2 d_2 \tilde{c}_2^{-1} b_2^{-1}.
\end{equation*}
Since $d_2^{-1}$ commutes with both $b_1$ and $b_2$
by the fourth and the seventeenth relations,
we deduce that
\begin{equation*}
1 = b_2^{-1} b_1^{-1} b_2 \tilde{c}_2 b_1 b_2 \tilde{c}_2^{-1} b_2^{-1} 
= b_2^{-1} b_1^{-1} b_2 b_2 b_1 b_2^{-1}.
\end{equation*}
Hence $b_1^{-1}b_2^2 b_1 b_2^{-2}=1$, and so 
$b_1$ commutes with $b_2^2$.

From the ninth and the tenth relations, we deduce that
\begin{align*}
c_1 &= b_2^{-1} d_1^{-1} b_2 d_1=b_2^{-1} [c_1^{-1},b_2] b_2 [b_2,c_1^{-1}]\\
&= b_2^{-1} c_1^{-1} b_2 c_1 b_2 c_1^{-1} b_2^{-1} c_1
= b_2^{-1} d_1^{-1} b_2 b_2 c_1^{-1} b_2^{-1} c_1.
\end{align*}
Canceling the $c_1$'s from both sides and then rearranging, we conclude that
\begin{equation*}
d_1=b_2^2 c_1^{-1} b_2^{-2}.
\end{equation*}
Since $b_1$ also commutes with $c_1$ by the fourteenth relation,
$b_1$ must commute with $d_1$.
It follows that $a_1=[b_1^{-1},d_1^{-1}]^n=1$.
From $a_1=1$, we can easily deduce that all other generators are trivial
except for $d_2$.
By Lemma~\ref{lem: infinite order}, $d_2$ has order $p$\/ in $\pi_1(\M)$ if $p$\/ is a positive integer.
\end{proof}

\begin{rmk}
At the moment,
$\sM$ has the smallest Euler characteristic amongst
all closed minimal symplectic\/ $4$-manifolds having fundamental group isomorphic
to $\Z/p$.
\end{rmk}

\section{Homeomorphism and Seiberg-Witten invariant}
\label{sec: sw}

Throughout this section, let $p=1$.
For every integer $n\geq1$, $\Mn$ is a closed simply connected spin $4$-manifold having intersection form $H$\/ (see (\ref{eq: H})) with a basis given by the homology classes of genus 2 surfaces $q(\Sigma_2 \times \{w_0\})$ and $q(\{z_0\} \times \Sigma_3)$.
By Freedman's classification theorem in \cite{freedman},
$\Mn$ is homeomorphic to $S^2\times S^2$ for every $n\geq 1$.

\begin{thm}\label{thm: M_1^1}
The symplectic\/ $4$-manifold $\ssM$ is homeomorphic but
not diffeomorphic to $S^2\times S^2$.
\end{thm}

\begin{proof}
From \cite{li}, we know that the symplectic Kodaira dimension is a diffeomorphism invariant.
The rational ruled surface $S^2\times S^2$ has Kodaira dimension $-\infty$.
$X$\/ is a minimal surface of general type and hence has Kodaira dimension $2$.
Since $\ssM$ is the result of six Luttinger surgeries on $X$\/ and Luttinger surgeries preserve symplectic Kodaira dimension (cf.~\cite{ho-li}),
$\ssM$ has Kodaira dimension $2$ as well.
\end{proof}

Let $A$\/ and $B$\/ denote the 2-dimensional cohomology classes of $\Mn$
that are Poincar\'e dual to the homology classes of $q(\Sigma_2 \times \{w_0\})$ and $q(\{z_0\} \times \Sigma_3)$, respectively.  Let
\begin{equation*}
SW_{\Mn}: H^2(\Mn;\Z) \longrightarrow \Z
\end{equation*}
denote the `small perturbation' Seiberg-Witten invariant of $\Mn$
(cf.~Lemma 3.2 in \cite{szabo}).

\begin{thm}\label{thm: sw}
$SW_{\Mn}(L)\neq 0$ only when $L=\pm(2A+2B)$,
and
\begin{equation*}
|SW_{\Mn}(\pm(2A+2B))|=n.
\end{equation*}
\end{thm}

\begin{proof}
Let $Z$\/ denote the symplectic $4$-manifold obtained by performing the following five Luttinger surgeries on $X$:
\begin{equation*}
\begin{array}{c}
(b_1' \times c_1'', b_1', -1), \ \
(a_2' \times c_1', c_1', -1),  \ \
(a_2'' \times d_1', d_1', -1),\\[2pt]
(b_1'' \times d_2', d_2', -1), \ \
((\tilde{b}_1\tilde{b}_2) \times \tilde{c}_2', \beta, -1).
\end{array}
\end{equation*}
Note that these are five of the six surgeries in (\ref{eq: 6 surgeries})
with $p=1$.
Hence we obtain $\Mn$ by performing
$(a_1' \times c_1', a_1', -n)$ surgery on $Z$.
We have $e(Z)=4$, $\sigma(Z)=0$,
$b_1(Z)=1$, $b_2(Z)=4$, and
the intersection form of $Z$\/ is isomorphic to $2H$\/ with a basis given by
\begin{equation}\label{eq: Z basis}
([a_1 \times c_1], -[b_1 \times d_1]),\ \
([\Sigma_2 \times \{w_0\}], [\{z_0\} \times \Sigma_3]).
\end{equation}
As shown in \cite{FPS, ABP}, our theorem will follow at once if we can prove that
the Seiberg-Witten invariant of $Z$\/ is nonzero only on
$\pm c_1(Z)$.

We abuse the notation slightly and
let $A$\/ and $B$\/ also denote the Poincar\'e duals of
$[\Sigma_2 \times \{w_0\}]$ and $[\{z_0\} \times \Sigma_3]$,
respectively, in $H^2(Z;\Z)$.
If $SW_Z (L) \neq 0$, then by applying the adjunction inequality
to four surfaces in (\ref{eq: Z basis}), we conclude that
$L= rA + sB$, where $r$\/ and $s$\/ are even integers satisfying
$|r|\leq 2$ and $|s|\leq 2$.
Since the dimension of the Seiberg-Witten moduli space for $L$\/ is nonnegative,
we must have
\begin{equation*}
L^2 = 2rs \geq 2e(Z)+3\sigma(Z)=8.
\end{equation*}
It follows that $r=s=\pm 2$, and $L=\pm(2A+2B)=\mp c_1(Z)$.
By Taubes's theorem in \cite{taubes}, we know that
$|SW_Z(\pm c_1(Z))|=1$.
\end{proof}

Since the value of the Seiberg-Witten invariant for the canonical class
of a symplectic $4$-manifold is always $\pm 1$ by the work of Taubes
\cite{taubes} (see the proof of Theorem~1.2 in \cite{szabo}
for the $b_2^+=1$ case),
$\Mn$ cannot be symplectic when $n\geq 2$.
Hence we conclude that $\{\Mn \mid n\geq 2\}$ are irreducible
(see Lemma~\ref{lem: irreducibility}), nonsymplectic and mutually nondiffeomorphic.
This concludes the proof of Theorem~\ref{thm: main}.

\begin{rmk}
In \cite{rasmussen}, Rasmussen has computed the Ozsv\'ath-Szab\'o invariant
of $\Mn$,
and has shown that $\Mn$'s do not admit any perfect Morse function.
\end{rmk}

\section*{Acknowledgments}
The first author was partially supported by an NSF grant.
The second author was partially supported by an NSERC discovery grant.
Some of the computations have been double-checked by
using the computer software package GAP \cite{GAP}.
The authors thank Selman Akbulut, Ronald Fintushel, Alexander Hulpke,
Masashi Ishida, Tian-Jun Li and Ronald J. Stern for helpful discussions.

\end{document}